\documentclass{amsart}%
\usepackage{amsmath}
\usepackage{amssymb}
\usepackage{amsfonts}
\usepackage{graphicx}%
\providecommand{\U}[1]{\protect\rule{.1in}{.1in}}
\newtheorem{theorem}{Theorem}
\theoremstyle{plain}

\newtheorem{corollary}{Corollary}

\newtheorem{lemma}{Lemma}

\newtheorem{proposition}{Proposition}
\newtheorem{remark}{Remark}

\numberwithin{equation}{section}
\begin{document}
\title[Continuants and convergence of certain continued fractions]{Continuants and convergence of certain continued fractions}
\author{Daniel Duverney}
\address{110, rue du chevalier fran\c{c}ais, 59000 Lille, France}
\email{daniel.duverney@orange.fr}
\author{Iekata Shiokawa}
\address{13-43, Fujizuka-cho, Hodogaya-ku, Yokohama 240-0031, Japan}
\email{shiokawa@beige.ocn.ne.jp}
\date{October 18, 2022}

\begin{abstract}
We give a concise introduction to the theory of continuants and show how
Perron used them in his proof of Tietze theorem on the convergence of infinite
semi-regular continued fractions, as well as for the study of the convergence
of purely periodic continued fractions.

\end{abstract}
\maketitle

\section{Introduction}

The purpose of this note is to give a short introduction to continuants and
show how they can be used in order to study the convergence of continued
fractions in some cases. Continuants and their relations to continued
fractions were first considered by Spottiswoode in 1856 \cite{Spot}. A few
years later, Nachreiner \cite{Nach} and Muir \cite{Muir} studied them
independently. Muir called them \textit{continuants} for the first
time.\smallskip

First we recall some basic facts about continued fractions \cite{JoTh}. Let
$a_{1},$ $a_{2},\ldots,$ $b_{0},$ $b_{1},$ $b_{2},\ldots$ be infinite
sequences of indeterminates. We define
\begin{equation}
A_{-1}=1,\quad A_{0}=b_{0},\quad B_{-1}=0,\quad B_{0}=1, \label{Int01}%
\end{equation}
and for $n\geq1$%
\begin{align}
A_{n}  &  =b_{n}A_{n-1}+a_{n}A_{n-2},\label{Int1}\\
B_{n}  &  =b_{n}B_{n-1}+a_{n}B_{n-2}. \label{Int2}%
\end{align}
It is well known that%
\begin{equation}
\alpha_{n}:=b_{0}+\frac{a_{1}}{b_{1}}%
\genfrac{}{}{0pt}{}{{}}{+}%
\frac{a_{2}}{b_{2}}%
\genfrac{}{}{0pt}{}{{}}{+\cdots+}%
\frac{a_{n}}{b_{n}}=\frac{A_{n}}{B_{n}}\qquad\left(  n\geq1\right)  ,
\label{Int6}%
\end{equation}
and an easy induction using (\ref{Int1}) and (\ref{Int2}) shows that%
\begin{equation}
A_{n}B_{n-1}-A_{n-1}B_{n}=\left(  -1\right)  ^{n-1}a_{1}a_{2}\cdots
a_{n}\qquad\left(  n\geq1\right)  . \label{Int7}%
\end{equation}
This yields immediately%
\begin{equation}
\frac{A_{n}}{B_{n}}-\frac{A_{n-1}}{B_{n-1}}=\left(  -1\right)  ^{n-1}%
\frac{a_{1}a_{2}\cdots a_{n}}{B_{n-1}B_{n}}\qquad\left(  n\geq1\right)  .
\label{K777}%
\end{equation}
Continuants allow to generalize (\ref{Int1}) and (\ref{Int2}) by expressing
$A_{n+k}$ and $B_{n+k}$ in terms of $A_{n-1},$ $A_{n-2},$ $B_{n-1},$ $B_{n-2}$
for all $k\geq0$. This leads consequently to a generalization of (\ref{Int7})
and (\ref{K777}), as we will see in Section \ref{SecCont}, which consists of a
short introduction to continuants.\smallskip

In Section \ref{SecTietze}, we will show how Perron used continuants in
\cite[Chapter 5]{Per} for proving Theorem \ref{TietzeTh} below, known as
\textit{Tietze theorem}. By Theorem \ref{TietzeTh}, any infinite semi-regular
continued fraction (\cite{Tietze},\cite{Per},\cite{DS}) is convergent. An
alternative proof of Theorem \ref{TietzeTh} (not using continuants) can be
found in \cite{DS2}.

\begin{theorem}
\label{TietzeTh}Assume that the infinite continued fraction
\begin{equation}
\alpha:=b_{0}+\frac{a_{1}}{b_{1}}%
\genfrac{}{}{0pt}{}{{}}{+}%
\frac{a_{2}}{b_{2}}%
\genfrac{}{}{0pt}{}{{}}{+\cdots+}%
\frac{a_{n}}{b_{n}}%
\genfrac{}{}{0pt}{}{{}}{+\cdots}
\label{Int21}%
\end{equation}
satisfies the following conditions:%
\begin{equation}
a_{n}\in\left\{  -1,1\right\}  ,\qquad b_{n}\in\left[  1,+\infty\right[
,\qquad b_{n}+a_{n+1}\geq1\qquad\left(  n\geq1\right)  . \label{Int22}%
\end{equation}
Then $\alpha$ is convergent.
\end{theorem}

In Section \ref{SecStolz}, we will use continuants for proving the following
result on the convergence of purely periodic continued fractions.

\begin{theorem}
\label{StolzTh}Let $\left(  a_{n}\right)  _{n\geq1}$ and $(b_{n})_{n\geq0}$ be
non-zero complex numbers. Assume that the infinite continued fraction%
\[
\alpha:=b_{0}+\frac{a_{1}}{b_{1}}%
\genfrac{}{}{0pt}{}{{}}{+}%
\frac{a_{2}}{b_{2}}%
\genfrac{}{}{0pt}{}{{}}{+\cdots+}%
\frac{a_{n}}{b_{n}}%
\genfrac{}{}{0pt}{}{{}}{+\cdots}%
\]
is purely periodic with period $p\geq1,$ that is
\[
a_{n+p}=a_{n}\quad\left(  n\geq1\right)  ,\qquad b_{n+p}=b_{n}\quad\left(
n\geq0\right)  .
\]
Let $\lambda_{1}$ and $\lambda_{2}$ with $\left\vert \lambda_{1}\right\vert
\geq\left\vert \lambda_{2}\right\vert $ be the eigenvalues of the matrix%
\[
M:=\left(
\begin{array}
[c]{cc}%
A_{p-1} & a_{p}A_{p-2}\\
B_{p-1} & a_{p}B_{p-2}%
\end{array}
\right)  =\left(
\begin{array}
[c]{cc}%
A_{p-1} & A_{p}-b_{0}A_{p-1}\\
B_{p-1} & B_{p}-b_{0}B_{p-1}%
\end{array}
\right)  .
\]
Then $\alpha$ is convergent if and only if $B_{p-1}\neq0$ and one of the
following conditions holds:

(C1) $\lambda_{1}=\lambda_{2}.$

(C2) $\left\vert \lambda_{1}\right\vert >\left\vert \lambda_{2}\right\vert $
and $A_{q}-x_{2}B_{q}\neq0$ for $0\leq q\leq p-2,$ where $x_{2}$ is defined by%
\begin{equation}
\lambda_{2}=x_{2}B_{p-1}+a_{p}B_{p-2}. \label{X2}%
\end{equation}

Moreover, in case of convergence, $\alpha=x_{1},$ where $x_{1}$ is defined by%
\begin{equation}
\lambda_{1}=x_{1}B_{p-1}+a_{p}B_{p-2}. \label{X1}%
\end{equation}

\end{theorem}

Theorem \ref{StolzTh} was first proved by Stolz \cite{Stolz}. The basic idea
of the proof we give in Section \ref{SecStolz} is due to Perron \cite[Page
83]{Per2}, although Perron doesn't use matrix calculations. It consist in
computing the values of $A_{n}$ and $B_{n}$ for all $n\geq0$ in function of
$\lambda_{1}$ and $\lambda_{2}.$ An alternative treatment, using the
properties of linear fractional transformations, can be found in \cite[Section
3.2]{JoTh}.\smallskip

Finally, in Section \ref{SecGalois1}, we will deduce from Theorem
\ref{StolzTh} a short proof of Theorem \ref{ThGal} below, known as
\textit{Galois generalized theorem} \cite[Theorem 3.4]{JoTh}. Recall that the
original Galois theorem applies to regular continued fractions (\cite{G}%
,\cite[Satz 3.6]{Per}, \cite[Exercise 4.6]{Duv}). Let
\begin{equation}
\alpha:=b_{0}+\frac{a_{1}}{b_{1}}%
\genfrac{}{}{0pt}{}{{}}{+\cdots+}%
\frac{a_{p-1}}{b_{p-1}}%
\genfrac{}{}{0pt}{}{{}}{+}%
\frac{a_{p}}{b_{0}}%
\genfrac{}{}{0pt}{}{{}}{+}%
\frac{a_{1}}{b_{1}}%
\genfrac{}{}{0pt}{}{{}}{+\cdots+}%
\frac{a_{p-1}}{b_{p-1}}%
\genfrac{}{}{0pt}{}{{}}{+}%
\frac{a_{p}}{b_{0}}%
\genfrac{}{}{0pt}{}{{}}{+\cdots}
\label{Alpha}%
\end{equation}
be a purely periodic continued fraction, and let%
\begin{equation}
\alpha^{\prime}:=b_{0}+\frac{a_{p}}{b_{p-1}}%
\genfrac{}{}{0pt}{}{{}}{+\cdots+}%
\frac{a_{2}}{b_{1}}%
\genfrac{}{}{0pt}{}{{}}{+}%
\frac{a_{1}}{b_{0}}%
\genfrac{}{}{0pt}{}{{}}{+}%
\frac{a_{p}}{b_{p-1}}%
\genfrac{}{}{0pt}{}{{}}{+\cdots+}%
\frac{a_{2}}{b_{1}}%
\genfrac{}{}{0pt}{}{{}}{+}%
\frac{a_{1}}{b_{0}}%
\genfrac{}{}{0pt}{}{{}}{+\cdots}%
, \label{AlphaP}%
\end{equation}
its reverse continued fraction. Note that $\alpha^{\prime}$ is purely
periodic. Define $a_{n}^{\prime}$ and $b_{n}^{\prime}$ $(n\geq1)$ by%
\[
\alpha^{\prime}=b_{0}^{\prime}+\frac{a_{1}^{\prime}}{b_{1}^{\prime}}%
\genfrac{}{}{0pt}{}{{}}{+}%
\frac{a_{2}^{\prime}}{b_{2}^{\prime}}%
\genfrac{}{}{0pt}{}{{}}{+\cdots+}%
\frac{a_{n}^{\prime}}{b_{n}^{\prime}}%
\genfrac{}{}{0pt}{}{{}}{+\cdots}%
\]
and $A_{n}^{\prime},$ $B_{n}^{\prime}$ $(n\geq-1)$ by $A_{-1}^{\prime}=1,$
$A_{0}^{\prime}=b_{0}^{\prime}=b_{0},$ $B_{-1}^{\prime}=0,$ $A_{-1}^{\prime
}=1$ and%
\[
A_{n}^{\prime}=b_{n}^{\prime}A_{n-1}^{\prime}+a_{n}^{\prime}A_{n-2}^{\prime
},\quad B_{n}^{\prime}=b_{n}^{\prime}B_{n-1}^{\prime}+a_{n}^{\prime}%
B_{n-2}^{\prime}\quad\left(  n\geq1\right)  .
\]

\begin{theorem}
\label{ThGal}Assume that the purely periodic continued fraction (\ref{Alpha})
is convergent, and let $\lambda_{1},$ $\lambda_{2},$ $x_{1},$ and $x_{2}$ be
as in Theorem \ref{StolzTh}. Then $\alpha=x_{1}$ and its reverse continued
fraction (\ref{AlphaP}) converges to $b_{0}-x_{2}$, except if $\left\vert
\lambda_{1}\right\vert >\left\vert \lambda_{2}\right\vert $ and there exists
$q\in\left\{  0,\ldots,p-2\right\}  $ such that%
\begin{equation}
A_{q}^{\prime}-\left(  b_{0}-x_{1}\right)  B_{q}^{\prime}=0 \label{Cond123}%
\end{equation}
in which case it is divergent.
\end{theorem}

\begin{corollary}
\label{CorGal}Assume that the $a_{n}$ $\left(  n\geq1\right)  $ and $b_{n}$
$(n\geq0)$ are non-zero integers and that the continued fraction (\ref{Alpha})
converges to an irrational $\alpha$\footnote{The example $\alpha:=2-\frac
{1}{2}%
\genfrac{}{}{0pt}{}{{}}{-}%
\frac{1}{2}%
\genfrac{}{}{0pt}{}{{}}{-\cdots-}%
\frac{1}{2}%
\genfrac{}{}{0pt}{}{{}}{-\cdots}%
=1$ shows that $\alpha$ can be rational.}. Then $\alpha$ is quadratic and%
\[
-\alpha^{\ast}=\alpha^{\prime}-b_{0}=\frac{a_{p}}{b_{p-1}}%
\genfrac{}{}{0pt}{}{{}}{+\cdots+}%
\frac{a_{2}}{b_{1}}%
\genfrac{}{}{0pt}{}{{}}{+}%
\frac{a_{1}}{b_{0}}%
\genfrac{}{}{0pt}{}{{}}{+}%
\frac{a_{p}}{b_{p-1}}%
\genfrac{}{}{0pt}{}{{}}{+\cdots+}%
\frac{a_{2}}{b_{1}}%
\genfrac{}{}{0pt}{}{{}}{+}%
\frac{a_{1}}{b_{0}}%
\genfrac{}{}{0pt}{}{{}}{+\cdots}%
,
\]
where $\alpha^{\ast}$ is the conjugate of $\alpha.$
\end{corollary}

Corollary \ref{CorGal} will also be proved in Section \ref{SecGalois1}. It
applies, for example, to any purely periodic semi-regular continued fraction
$\alpha.$ In the case of a regular continued fraction, i.e. $a_{n}=1$ for all
$n\geq1,$%
\[
-\frac{1}{\alpha^{\ast}}=b_{p-1}+\frac{1}{b_{p-2}}%
\genfrac{}{}{0pt}{}{{}}{+\cdots+}%
\frac{1}{b_{1}}%
\genfrac{}{}{0pt}{}{{}}{+}%
\frac{1}{b_{0}}%
\genfrac{}{}{0pt}{}{{}}{+}%
\frac{1}{b_{p-1}}%
\genfrac{}{}{0pt}{}{{}}{+\cdots+}%
\frac{1}{b_{1}}%
\genfrac{}{}{0pt}{}{{}}{+}%
\frac{1}{b_{0}}%
\genfrac{}{}{0pt}{}{{}}{+\cdots}%
.
\]
This is Galois Theorem. Similarly, in the case of a negative continued
fraction, i.e. $a_{n}=-1$ for all $n\geq1$,%
\[
\frac{1}{\alpha^{\ast}}=b_{p-1}-\frac{1}{b_{p-2}}%
\genfrac{}{}{0pt}{}{{}}{-\cdots-}%
\frac{1}{b_{1}}%
\genfrac{}{}{0pt}{}{{}}{-}%
\frac{1}{b_{0}}%
\genfrac{}{}{0pt}{}{{}}{-}%
\frac{1}{b_{p-1}}%
\genfrac{}{}{0pt}{}{{}}{-\cdots-}%
\frac{1}{b_{1}}%
\genfrac{}{}{0pt}{}{{}}{-}%
\frac{1}{b_{0}}%
\genfrac{}{}{0pt}{}{{}}{+\cdots}%
,
\]
which has been proved by M\"{o}bius (\cite{Mo},\cite[Satz 19]{Z}).

\section{Continuants}

\label{SecCont}We call \textit{continuant} \cite{Muir} any determinant of the
form%
\begin{equation}
K\binom{a_{1},\ldots,a_{n}}{b_{0},\ldots,b_{n}}:=\left\vert
\begin{array}
[c]{ccccc}%
b_{0} & -1 & 0 & \cdots & 0\\
a_{1} & b_{1} & -1 & \ddots & \vdots\\
0 & a_{2} & b_{2} & \ddots & 0\\
\vdots &  & \ddots & \ddots & -1\\
0 & \cdots & 0 & a_{n} & b_{n}%
\end{array}
\right\vert \qquad\left(  n\geq0\right)  , \label{K1}%
\end{equation}
with for $n=0$ the notation%
\begin{equation}
K\binom{\ast}{b_{0}}:=b_{0}. \label{K2}%
\end{equation}
Developing with respect to the last line yields for $n\geq0$%
\begin{equation}
K\binom{a_{1},\ldots,a_{n+2}}{b_{0},\ldots,b_{n+2}}=b_{n+2}K\binom
{a_{1},\ldots,a_{n+1}}{b_{0},\ldots,b_{n+1}}+a_{n+2}K\binom{a_{1},\ldots
,a_{n}}{b_{0},\ldots,b_{n}}, \label{K3}%
\end{equation}
Developing with respect to the first column yields for $n\geq0$%
\begin{equation}
K\binom{a_{1},\ldots,a_{n+2}}{b_{0},\ldots,b_{n+2}}=b_{0}K\binom{a_{2}%
,\ldots,a_{n+2}}{b_{1},\ldots,b_{n+2}}+a_{1}K\binom{a_{3},\ldots,a_{n+2}%
}{b_{2},\ldots,b_{n+2}}. \label{K4}%
\end{equation}
Let $A_{n}$ and $B_{n}$ be defined in (\ref{Int01}), (\ref{Int1}) and
(\ref{Int2}). It is clear from (\ref{K1}) that%
\begin{align}
K\binom{\ast}{b_{0}}  &  =b_{0}=A_{0},\qquad K\binom{a_{1}}{b_{0},b_{1}}%
=a_{1}+b_{0}b_{1}=A_{1},\label{K55}\\
K\binom{\ast}{b_{1}}  &  =b_{1}=B_{1},\qquad K\binom{a_{2}}{b_{1},b_{2}}%
=a_{2}+b_{1}b_{2}=B_{2}. \label{K56}%
\end{align}
Hence an easy induction using (\ref{K3}) shows that%
\begin{align}
A_{n}  &  =K\binom{a_{1},\ldots,a_{n}}{b_{0},\ldots,b_{n}}\qquad\left(
n\geq0\right)  ,\label{K8}\\
B_{n}  &  =K\binom{a_{2},\ldots,a_{n}}{b_{1},\ldots,b_{n}}\qquad\left(
n\geq1\right)  . \label{K9}%
\end{align}
As a first application, we have by following \cite[Page 9]{Per}:

\begin{proposition}
\label{PropInv}Let $n\geq0.$ Define $A_{n},$ $B_{n},$ $A_{n}^{\prime}$ and
$B_{n}^{\prime}$ by%
\begin{align*}
\frac{A_{n}}{B_{n}}  &  =b_{0}+\frac{a_{1}}{b_{1}}%
\genfrac{}{}{0pt}{}{{}}{+}%
\frac{a_{2}}{b_{2}}%
\genfrac{}{}{0pt}{}{{}}{+\cdots+}%
\frac{a_{n}}{b_{n}},\\
\frac{A_{n}^{\prime}}{B_{n}^{\prime}}  &  =b_{n}+\frac{a_{n}}{b_{n-1}}%
\genfrac{}{}{0pt}{}{{}}{+}%
\frac{a_{n-1}}{b_{n-2}}%
\genfrac{}{}{0pt}{}{{}}{+\cdots+}%
\frac{a_{1}}{b_{0}}.
\end{align*}
Then for all $n\geq1,$%
\begin{equation}
A_{n}^{\prime}=A_{n},\quad B_{n}^{\prime}=A_{n-1},\quad A_{n-1}^{\prime}%
=B_{n},\quad B_{n-1}^{\prime}=B_{n-1}. \label{Rel}%
\end{equation}

\end{proposition}

\begin{proof}
We have by (\ref{K8}) and (\ref{K9})%
\begin{align*}
A_{n}^{\prime}  &  =K\binom{a_{n},\ldots,a_{1}}{b_{n},\ldots,b_{0}}%
=K\binom{a_{1},\ldots,a_{n}}{b_{0},\ldots,b_{n}}=A_{n},\\
B_{n}^{\prime}  &  =K\binom{a_{n-1},\ldots,a_{1}}{b_{n-1},\ldots,b_{0}%
}=K\binom{a_{1},\ldots,a_{n-1}}{b_{0},\ldots,b_{n-1}}=A_{n-1},\\
A_{n-1}^{\prime}  &  =K\binom{a_{n},\ldots,a_{2}}{b_{n},\ldots,b_{1}}%
=K\binom{a_{2},\ldots,a_{n}}{b_{1},\ldots,b_{n}}=B_{n},\\
B_{n-1}^{\prime}  &  =K\binom{a_{n-1},\ldots,a_{2}}{b_{n-1},\ldots,b_{1}%
}=K\binom{a_{2},\ldots,a_{n-1}}{b_{1},\ldots,b_{n-1}}=B_{p-1},
\end{align*}
since a determinant is unchanged by symmetry with respect to its second
diagonal\footnote{Up to now, I didn't know this result... I could find by
myself an elementary proof of it, but I have no reference. Do you know one?}.
\end{proof}

More generally, define%
\begin{equation}
\alpha_{k,n}:=b_{k}+\frac{a_{k+1}}{b_{k+1}}%
\genfrac{}{}{0pt}{}{{}}{+\cdots+}%
\frac{a_{k+n}}{b_{k+n}}=\frac{A_{k,n}}{B_{k,n}}\qquad\left(  k\geq
0,n\geq1\right)  . \label{K10}%
\end{equation}
Then $\alpha_{0,n}=\alpha_{n},$ $A_{0,n}=A_{n},$ $B_{0,n}=B_{n}$ and
\begin{equation}
A_{k,-1}=1,\quad A_{k,0}=b_{k},\quad B_{k,-1}=0,\quad B_{k,0}=1. \label{K5}%
\end{equation}
Moreover, by (\ref{K8}) and (\ref{K9})
\begin{align}
A_{k,n}  &  =K\binom{a_{k+1},\ldots,a_{k+n}}{b_{k},\ldots,b_{k+n}}%
\qquad\left(  k\geq0,n\geq0\right)  ,\label{K88}\\
B_{k,n}  &  =K\binom{a_{k+2},\ldots,a_{k+n}}{b_{k+1},\ldots,b_{k+n}}%
\qquad\left(  k\geq0,n\geq1\right)  . \label{K89}%
\end{align}
As $B_{k,0}=A_{k+1,-1}=1,$ this yields immediately%
\begin{equation}
B_{k,n}=A_{k+1,n-1}\qquad\left(  k\geq0,n\geq0\right)  . \label{K11}%
\end{equation}
Clearly $A_{k,1}=b_{k}A_{k+1,0}+a_{k+1}A_{k+2,-1}$ by (\ref{K55}) and
(\ref{K5}). Therefore replacing $a_{n}$ and $b_{n}$ by $a_{k+n}$ and $b_{k+n}$
in (\ref{K4}) yields
\begin{equation}
A_{k,n+2}=b_{k}A_{k+1,n+1}+a_{k+1}A_{k+2,n}\qquad\left(  k\geq0,n\geq
-1\right)  . \label{K15}%
\end{equation}
So by using (\ref{K11}), (\ref{K56}) and (\ref{K5})
\begin{equation}
B_{k,n+2}=b_{k+1}B_{k+1,n+1}+a_{k+2}B_{k+2,n}\qquad\left(  k\geq
0,n\geq-1\right)  . \label{K15b}%
\end{equation}

\begin{proposition}
\label{PropOffer}If $k\geq0$ and $n\geq1,$ then%
\begin{align}
A_{n+k}  &  =A_{n,k}A_{n-1}+a_{n}B_{n,k}A_{n-2},\label{K100}\\
B_{n+k}  &  =A_{n,k}B_{n-1}+a_{n}B_{n,k}B_{n-2}, \label{K101}%
\end{align}

\end{proposition}

\begin{proof}
We follow \cite[Prop.1]{Offer}. Let $m=n+k.$ We have to prove that%
\begin{align}
A_{m}  &  =A_{n,m-n}A_{n-1}+a_{n}B_{n,m-n}A_{n-2}\qquad\left(  1\leq n\leq
m\right)  ,\label{K17}\\
B_{m}  &  =A_{n,m-n}B_{n-1}+a_{n}B_{n,m-n}B_{n-2}\qquad\left(  1\leq n\leq
m\right)  . \label{K18}%
\end{align}
Let $f(n):=A_{n,m-n}A_{n-1}+a_{n}B_{n,m-n}A_{n-2}.$ Then for $1\leq n<m$%
\begin{align*}
f(n+1)  &  =A_{n+1,m-n-1}A_{n}+a_{n+1}B_{n+1,m-n-1}A_{n-1}\\
&  =A_{n+1,m-n-1}(b_{n}A_{n-1}+a_{n}A_{n-2})+a_{n+1}B_{n+1,m-n-1}A_{n-1}\\
&  =\left(  b_{n}A_{n+1,m-n-1}+a_{n+1}B_{n+1,m-n-1}\right)  A_{n-1}%
+a_{n}A_{n+1,m-n-1}A_{n-2}\\
&  =\left(  b_{n}A_{n+1,m-n-1}+a_{n+1}A_{n+2,m-n-2}\right)  A_{n-1}%
+a_{n}B_{n,m-n}A_{n-2}\\
&  =A_{n,m-n}A_{n-1}+a_{n}B_{n,m-n}A_{n-2}=f(n).
\end{align*}
Hence $f(n)=f(m)=A_{n,0}A_{n-1}+a_{n}B_{n,0}A_{n-2}=b_{n}A_{n-1}+a_{n}%
A_{n-2}=A_{n},$ which proves (\ref{K17}). The proof of (\ref{K18}) is similar:
just replace $A_{n},$ $A_{n-1}$ and $A_{n-2}$ by $B_{n},$ $B_{n-1}$ and
$B_{n-2}$ respectively.
\end{proof}

As $A_{n,0}=b_{n}$ and $B_{n,0}=1$ by (\ref{K5}), (\ref{Int1}) and
(\ref{Int2}) result from (\ref{K100}) and (\ref{K101}) respectively by taking
$k=0$.

\begin{proposition}
\label{PropPerron}Let $n\geq1$ and $k\geq0.$ Then%
\begin{equation}
A_{n+k}B_{n-1}-A_{n-1}B_{n+k}=\left(  -1\right)  ^{n-1}a_{1}a_{2}\cdots
a_{n}B_{n,k}. \label{K200}%
\end{equation}

\end{proposition}

\begin{proof}
For $n\geq1$ and $k\geq0$, let $h(n,k):=A_{n+k}B_{n-1}-A_{n-1}B_{n+k}.$Then by
(\ref{K100}) and (\ref{K101})%
\begin{align*}
h(n,k)  &  =\left(  A_{n,k}A_{n-1}+a_{n}B_{n,k}A_{n-2}\right)  B_{n-1}%
-A_{n-1}\left(  A_{n,k}B_{n-1}+a_{n}B_{n,k}B_{n-2}\right) \\
&  =a_{n}B_{n,k}\left(  A_{n-2}B_{n-1}-A_{n-1}B_{n-2}\right)  .
\end{align*}
If $n=1,$ then (\ref{K200}) is true by (\ref{Int01}). If $n\geq2,$
(\ref{K200}) results from (\ref{Int7}).
\end{proof}

From (\ref{K200}) we deduce immediately the following generalization of
(\ref{K777}):%
\begin{equation}
\frac{A_{n+k}}{B_{n+k}}-\frac{A_{n-1}}{B_{n-1}}=\left(  -1\right)  ^{n-1}%
a_{1}a_{2}\cdots a_{n}\frac{B_{n,k}}{B_{n-1}B_{n+k}}\qquad\left(  n\geq
1,k\geq0\right)  . \label{K201}%
\end{equation}
Since $B_{n,0}=1$ by (\ref{K5}), (\ref{Int7}) and (\ref{K777}) result from
(\ref{K200}) and (\ref{K201}) respectively by taking $k=0$. As observed by
Perron, for $k=1$ we obtain by (\ref{K11}) and (\ref{K5})%
\begin{equation}
\frac{A_{n+1}}{B_{n+1}}-\frac{A_{n-1}}{B_{n-1}}=\left(  -1\right)  ^{n-1}%
a_{1}a_{2}\cdots a_{n}\frac{b_{n+1}}{B_{n-1}B_{n+1}}\qquad\left(
n\geq1\right)  . \label{K300}%
\end{equation}
This formula can also be deduced directly from (\ref{Int7}) and (\ref{Int2})
by writing%
\begin{align*}
\frac{A_{n+1}}{B_{n+1}}-\frac{A_{n-1}}{B_{n-1}}  &  =\frac{A_{n+1}}{B_{n+1}%
}-\frac{A_{n}}{B_{n}}+\frac{A_{n}}{B_{n}}-\frac{A_{n-1}}{B_{n-1}}\\
&  =\left(  -1\right)  ^{n-1}a_{1}a_{2}\cdots a_{n}\left(  \frac{1}%
{B_{n-1}B_{n}}-\frac{a_{n+1}}{B_{n}B_{n+1}}\right) \\
&  =\left(  -1\right)  ^{n-1}a_{1}a_{2}\cdots a_{n}\frac{B_{n+1}%
-a_{n+1}B_{n-1}}{B_{n-1}B_{n}B_{n+1}}.
\end{align*}

\section{Perron's proof of Tietze theorem}

\label{SecTietze}It makes use of two lemmas.

\begin{lemma}
\label{LemmaT2}Let $\alpha$ be the semi-regular continued fraction defined by
(\ref{Int21}). Then%
\begin{equation}
1\leq B_{k,n}\leq B_{k-1,n+1}\leq B_{k+n}\qquad\left(  k\geq1,n\geq0\right)  .
\label{K250}%
\end{equation}

\end{lemma}

\begin{proof}
\cite[p.149-150]{Per} Replacing $k$ by $k-1$ and $n$ by $n-1$ in (\ref{K15b})
yields%
\begin{equation}
B_{k-1,n+1}-B_{k,n}=\left(  b_{k}-1\right)  B_{k,n}+a_{k+1}B_{k+1,n-1}%
\qquad\left(  k\geq1,n\geq0\right)  . \label{KKK}%
\end{equation}
We prove first by induction on $n$ that $B_{k,n}\geq0$ for $k\geq0$ and
$n\geq-1.$ This is true for $n=-1$ and $n=0$ for all $k\geq0$ by (\ref{K5}).
Assuming that it is true for $n-1$ and $n,$ we get by (\ref{Int22})%
\[
B_{k-1,n+1}-B_{k,n}\geq\left(  b_{k}-1\right)  \left(  B_{k,n}-B_{k+1,n-1}%
\right)  \qquad\left(  k\geq1\right)  ,
\]
which proves that $B_{k,n+1}\geq0$ for all $k\geq0.$ Hence we have
\begin{equation}
B_{k-1,n+1}-B_{k,n}\geq\left(  b_{k}-1\right)  \left(  B_{k,n}-B_{k+1,n-1}%
\right)  \qquad\left(  k\geq1,n\geq0\right)  . \label{KK}%
\end{equation}
However $B_{k+n,0}=1$ and $B_{k+n+1,-1}=0$ by (\ref{K5}), so that by an easy
induction
\[
B_{k-1,n+1}\geq B_{k,n}\qquad\left(  k\geq1,n\geq0\right)  .
\]
This proves (\ref{K250}) since $B_{k+n,0}=1$ and $B_{0,n+k}=B_{n+k}.$
\end{proof}

\begin{lemma}
\label{LemmaT1}Let $\alpha$ be the semi-regular continued fraction defined by
(\ref{Int21}). Then%
\[
\lim_{n\rightarrow\infty}B_{n}=+\infty.
\]

\end{lemma}

\begin{proof}
\cite[p.150-151]{Per} Let $k\geq1$ be fixed. If $a_{k+1}=1,$ then by
(\ref{KKK})%
\begin{equation}
B_{k-1,n+1}-B_{k,n}\geq B_{k+1,n-1}\geq1\qquad\left(  n\geq1\right)  .
\label{KKK2}%
\end{equation}
On the other hand, if $a_{k+1}=-1,$ then by (\ref{KKK}) and (\ref{Int22})%
\begin{equation}
B_{k-1,n+1}-B_{k,n}\geq B_{k,n}-B_{k+1,n-1}\qquad\left(  n\geq0\right)  .
\label{KKK3}%
\end{equation}
So in both cases%
\begin{equation}
B_{k-1,n+1}-B_{k,n}\geq\min\left(  1,B_{k,n}-B_{k+1,n-1}\right)  .
\label{KKK4}%
\end{equation}
Consequently, if $a_{k+1}=1,$ by an easy induction we see that%
\[
B_{k-m,n+m}-B_{k-m+1,n+m-1}\geq1\qquad\left(  1\leq m\leq k\right)  .
\]
Summing for $m=1$ to $k$ yields $B_{0,n+k}-B_{k,n}\geq k$ and therefore%
\begin{equation}
a_{k+1}=1\quad\Rightarrow\quad B_{k+n}\geq k+1\qquad\left(  k\geq
1,n\geq1\right)  . \label{KKK10}%
\end{equation}
If $a_{k+1}=-1,$ then $B_{k-1,1}-B_{k,0}\geq1$ by taking $n=0$ in
(\ref{KKK3}). Therefore by (\ref{KKK4})%
\[
B_{k-m,m}-B_{k-m+1,m-1}\geq1\qquad\left(  1\leq m\leq k\right)  .
\]
Hence summing again for $m=1$ to $k$ yields $B_{0,k}-B_{k,0}\geq k,$ so that
\begin{equation}
a_{k+1}=-1\quad\Rightarrow\quad B_{k}\geq k+1\qquad\left(  k\geq1\right)  .
\label{KKK20}%
\end{equation}
Now if $a_{k}=-1$ for all large $n,$ then clearly $\lim_{k\rightarrow\infty
}B_{k}=+\infty$ by (\ref{KKK20}). On the other hand, if there exist infinitely
many $k$ such that $a_{k}=1,$ then $\lim_{k\rightarrow\infty}B_{k}=+\infty$ by
(\ref{KKK10}) and the proof of Lemma \ref{LemmaT1} is complete.\smallskip
\end{proof}

Now we prove Tietze theorem by following \cite[p.151]{Per}. By (\ref{K201})
and (\ref{K250}), we have%
\[
\left\vert \frac{A_{n+k}}{B_{n+k}}-\frac{A_{n-1}}{B_{n-1}}\right\vert
\leq\frac{B_{n,k}}{B_{n-1}B_{n+k}}\leq\frac{1}{B_{n-1}}\qquad\left(
n\geq1,k\geq0\right)  .
\]
As $\lim_{n\rightarrow\infty}B_{n}=\infty$ by Lemma \ref{LemmaT1},
$A_{n}/B_{n}$ is a Cauchy sequence, which proves Tietze theorem.\smallskip

\begin{remark}
Two different alternative proofs of Lemma \ref{LemmaT1}, not using
continuants, can be found in \cite[Lem.3]{DS2} and in \cite[Th.1(i)]{Offer}.
\end{remark}

\section{Proof of Theorem \ref{StolzTh}}

\label{SecStolz}We will need the following lemma, which is closely connected
to the method of power iteration in numerical analysis.

\begin{lemma}
\label{LemPower}Let $\left(  u_{0},v_{0}\right)  \in\mathbb{C}^{2}%
\smallsetminus\left(  0,0\right)  $ and%
\[
M=\left(
\begin{array}
[c]{cc}%
a & c\\
b & d
\end{array}
\right)  ,\quad\left(  a,b,c,d\right)  \in\mathbb{C}^{4},\quad b\neq0,\quad
ad-bc\neq0.
\]
Let $\lambda_{1}$ and $\lambda_{2}$ with $\left\vert \lambda_{1}\right\vert
\geq\left\vert \lambda_{2}\right\vert $ be the eigenvalues of $M,$ and let
$(u_{n},v_{n})$ be defined recursively by%
\[
\left(
\begin{array}
[c]{c}%
u_{n+1}\\
v_{n+1}%
\end{array}
\right)  =M\left(
\begin{array}
[c]{c}%
u_{n}\\
v_{n}%
\end{array}
\right)  \quad\left(  n\geq0\right)  .
\]
Define $\left(  x_{1},x_{2}\right)  $ by $\lambda_{1}=bx_{1}+d$ and
$\lambda_{2}=bx_{2}+d$. Then:\smallskip

(i) If $\left\vert \lambda_{1}\right\vert >\left\vert \lambda_{2}\right\vert
,$ $v_{n}\neq0$ for all large $n$ and%
\begin{align}
\lim_{n\rightarrow\infty}\frac{u_{n}}{v_{n}}  &  =x_{1}\quad\text{if\quad
}u_{0}-v_{0}x_{2}\neq0,\label{C1}\\
\lim_{n\rightarrow\infty}\frac{u_{n}}{v_{n}}  &  =x_{2}\quad\text{if\quad
}u_{0}-v_{0}x_{2}=0. \label{C2}%
\end{align}

(ii) If $\left\vert \lambda_{1}\right\vert =\left\vert \lambda_{2}\right\vert
$ and $\lambda_{1}\neq\lambda_{2},$ $u_{n}/v_{n}$ is divergent, except if
\[
u_{0}-v_{0}x_{1}=0\quad or\quad u_{0}-v_{0}x_{2}=0.
\]

(iii) If $\lambda_{1}=\lambda_{2},$ $v_{n}\neq0$ for all large $n$ and
$\lim_{n\rightarrow\infty}u_{n}/v_{n}=x_{1}.$
\end{lemma}

\begin{proof}
We note first that $\lambda_{1}\neq0$ and $\lambda_{2}\neq0$ since $\det
M=ad-bc\neq0.$ Let $f$ be the linear transformation of $\mathbb{C}^{2}$ whose
matrix in the standard basis is $M.$ Let $e_{1},$ $e_{2},$ and $w_{n}$
$\left(  n\geq0\right)  $ be defined by%
\[
e_{1}=\left(
\begin{array}
[c]{c}%
x_{1}\\
1
\end{array}
\right)  ,\quad e_{2}=\left(
\begin{array}
[c]{c}%
x_{2}\\
1
\end{array}
\right)  ,\quad w_{n}=\left(
\begin{array}
[c]{c}%
u_{n}\\
v_{n}%
\end{array}
\right)  .
\]
It is easy to check that $Me_{1}=\lambda_{1}e_{1}$ and $Me_{2}=\lambda
_{2}e_{2}.$ Hence $e_{1}$ and $e_{2}$ are eigenvectors of $M$ and $f$ and they
form a basis of $\mathbb{C}^{2}$ if $\lambda_{1}\neq\lambda_{2}$. In this
case, define $\mu_{1}$ and $\mu_{2}$ by%
\begin{equation}
w_{0}=\mu_{1}e_{1}+\mu_{2}e_{2}. \label{Init}%
\end{equation}

We observe that%
\[
\mu_{1}=0\Leftrightarrow\left\vert
\begin{array}
[c]{cc}%
u_{0} & x_{2}\\
v_{0} & 1
\end{array}
\right\vert =0\Leftrightarrow u_{0}-v_{0}x_{2}=0
\]
and similarly $\mu_{2}=0\Leftrightarrow u_{0}-v_{0}x_{1}=0.$ It results from
(\ref{Init}) that
\begin{equation}
w_{n}=f^{n}\left(  w_{0}\right)  =\mu_{1}\lambda_{1}^{n}e_{1}+\mu_{2}%
\lambda_{2}^{n}e_{2}\quad\left(  n\geq0\right)  . \label{Wn}%
\end{equation}
Assume that $\left\vert \lambda_{1}\right\vert >\left\vert \lambda
_{2}\right\vert .$ If $\mu_{1}\neq0,$ we can write%
\[
w_{n}=\mu_{1}\lambda_{1}^{n}\left(  e_{1}+\frac{\mu_{2}}{\mu_{1}}\left(
\frac{\lambda_{2}}{\lambda_{1}}\right)  ^{n}e_{2}\right)  ,
\]
which proves that $v_{n}\neq0$ for all large $n$ and (\ref{C1}) holds. If
$\mu_{1}=0,$ then (\ref{C2}) holds by (\ref{Wn}), which proves \textit{(i).
}Now assume that $\left\vert \lambda_{1}\right\vert =\left\vert \lambda
_{2}\right\vert $ and $\lambda_{1}\neq\lambda_{2}.$ Then $\lambda_{2}%
=\lambda_{1}e^{i\theta}$ for some $\theta\in\left]  0,2\pi\right[  $ and by
(\ref{Wn}) we see that%
\begin{equation}
w_{n}=\lambda_{1}^{n}\left(  \mu_{1}e_{1}+\mu_{2}e^{in\theta}e_{2}\right)
\quad\left(  n\geq0\right)  . \label{Wn2}%
\end{equation}
So if $\mu_{1}=0$ or $\mu_{2}=0$ the sequence $u_{n}/v_{n}$ is well-defined
and convergent. If $\mu_{1}\neq0$ and $\mu_{2}\neq0,$ then either $v_{n}=0$
for infinitely many $n$ and $u_{n}/v_{n}$ is not defined, or%
\[
\frac{u_{n}}{v_{n}}=\frac{\mu_{1}x_{1}+\mu_{2}e^{in\theta}x_{2}}{\mu_{1}%
+\mu_{2}e^{in\theta}}=x_{1}+\left(  x_{2}-x_{1}\right)  \frac{\mu_{2}}{\mu
_{1}e^{-in\theta}+\mu_{2}},
\]
which is divergent since $x_{2}\neq x_{1},$ $\mu_{1}\neq0,$ $\mu_{2}\neq0$ and
$\theta\in\left]  0,2\pi\right[  .$ This proves \textit{(ii)}. Finally, if
$\lambda_{1}=\lambda_{2},$ there exists a basis $\left(  e_{1},e_{2}^{\prime
}\right)  $ of $\mathbb{C}^{2}$ such that $f(e_{2}^{\prime})=e_{1}+\lambda
_{1}e_{2}^{\prime}$ and so
\begin{equation}
w_{0}=\mu_{1}e_{1}+\mu_{2}e_{2}^{\prime} \label{Init1}%
\end{equation}
for some $(\mu_{1},\mu_{2})\in\mathbb{C}^{2}$. As $f^{n}(e_{2}^{\prime
})=n\lambda_{1}^{n-1}e_{1}+\lambda_{1}^{n}e_{2}^{\prime},$ we get%
\[
w_{n}=f^{n}(w_{0})=\lambda_{1}^{n-1}\left(  \left(  \mu_{1}\lambda_{1}%
+n\mu_{2}\right)  e_{1}+\lambda_{1}\mu_{2}e_{2}^{\prime}\right)  .
\]
If $\mu_{2}=0,$ we see that $u_{n}/v_{n}=x_{1}.$ If $\mu_{2}\neq0,$%
\[
\lim_{n\rightarrow\infty}\frac{u_{n}}{v_{n}}=\lim_{n\rightarrow\infty}%
\frac{\left(  \mu_{1}\lambda_{1}+n\mu_{2}\right)  x_{1}+\lambda_{1}\mu
_{2}x_{2}}{\left(  \mu_{1}\lambda_{1}+n\mu_{2}\right)  +\lambda_{1}\mu_{2}%
}=x_{1}%
\]
and \textit{(iii)} is proved.\smallskip
\end{proof}

Now we prove Theorem \ref{StolzTh}. Let $q\in\left\{  0,\ldots,p-1\right\}  .$
Replacing $n$ by $p$ and $k$ by $np+q$ in (\ref{K100}) and (\ref{K101}) yields%
\begin{align*}
A_{\left(  n+1\right)  p+q}  &  =A_{p-1}A_{p,np+q}+a_{p}A_{p-2}B_{p,np+q},\\
B_{\left(  n+1\right)  p+q}  &  =B_{p-1}A_{p,np+q}+a_{p}B_{p-2}B_{p,np+q}.
\end{align*}
However by (\ref{K88}) and (\ref{K8}) we have since $a_{n}$ and $b_{n}$ are
\textit{p}-periodic%
\[
A_{p,np+q}=K\binom{a_{p+1},\ldots,a_{\left(  n+1\right)  p+q}}{b_{p}%
,\ldots,b_{\left(  n+1\right)  p+q}}=K\binom{a_{1},\ldots,a_{np+q}}%
{b_{0},\ldots,b_{np+q}}=A_{np+q}.
\]
Similarly, by (\ref{K89}) and (\ref{K9}), we see that $B_{p,np+q}=B_{np+q}.$
Hence for $0\leq q\leq p-1$%
\begin{align}
A_{\left(  n+1\right)  p+q}  &  =A_{p-1}A_{np+q}+a_{p}A_{p-2}B_{np+q}%
,\label{S1}\\
B_{\left(  n+1\right)  p+q}  &  =B_{p-1}A_{np+q}+a_{p}B_{p-2}B_{np+q}.
\label{S2}%
\end{align}
\qquad

Assume that $B_{p-1}=0.$ Then by (\ref{S2}) with $q=p-1$ we have%
\[
B_{\left(  n+1\right)  p+p-1}=a_{p}B_{p-2}B_{np+p-1}=\left(  a_{p}%
B_{p-2}\right)  ^{n+1}B_{p-1}=0,
\]
so that the sequence $B_{n}$ vanishes infinitely often and therefore $\alpha$
is divergent.\smallskip

Assume that $B_{p-1}\neq0.$ Then (\ref{S1}) and (\ref{S2}) can be written as%
\[
\left(
\begin{array}
[c]{c}%
A_{\left(  n+1\right)  p+q}\\
B_{\left(  n+1\right)  p+q}%
\end{array}
\right)  =\left(
\begin{array}
[c]{cc}%
A_{p-1} & a_{p}A_{p-2}\\
B_{p-1} & a_{p}B_{p-2}%
\end{array}
\right)  \left(
\begin{array}
[c]{c}%
A_{np+q}\\
B_{np+q}%
\end{array}
\right)  =M\left(
\begin{array}
[c]{c}%
A_{np+q}\\
B_{np+q}%
\end{array}
\right)  .
\]
We observe that $\det M=\left(  -1\right)  ^{p}a_{1}\cdots a_{p}\neq0,$ so
that we can apply Lemma \ref{LemPower} with $u_{n}=u_{n,q}=A_{np+q}$ and
$v_{n}=v_{n,q}=B_{np+q}$ for every $q\in\left\{  0,\ldots,p-1\right\}  ,$ in
such a way that%
\[
u_{0,q}-v_{0,q}x_{1}=A_{q}-x_{1}B_{q},\quad u_{0,q}-v_{0,q}x_{2}=A_{q}%
-x_{2}B_{q}.
\]
As $\lambda_{1}+\lambda_{2}=A_{p-1}+a_{p}B_{p-2},$ for $q=p-1$ we see that%
\begin{align}
u_{0,p-1}-v_{0,p-1}x_{1}  &  =A_{p-1}-x_{1}B_{p-1}=A_{p-1}-\left(  \lambda
_{1}-a_{p}B_{p-2}\right)  =\lambda_{2}\neq0,\label{Init11}\\
u_{0,p-1}-v_{0,p-1}x_{2}  &  =A_{p-1}-x_{2}B_{p-1}=A_{p-1}-\left(  \lambda
_{2}-a_{p}B_{p-2}\right)  =\lambda_{1}\neq0. \label{Init12}%
\end{align}
We distinguish three cases.\smallskip

\textit{Case 1.} $\lambda_{1}=\lambda_{2}.$ By Lemma \ref{LemPower}
\textit{(iii)} we see that $\lim_{n\rightarrow\infty}A_{np+q}/B_{np+q}=x_{1}$
for every $q\in\left\{  0,\ldots,p-1\right\}  ,$ so that $\alpha$ is
convergent and $\alpha=x_{1}.\smallskip$

\textit{Case 2.} $\left\vert \lambda_{1}\right\vert =\left\vert \lambda
_{2}\right\vert $ and $\lambda_{1}\neq\lambda_{2}.$ Then the sequence
$A_{np+p-1}/B_{np+p-1}$ is divergent by Lemma \ref{LemPower} \textit{(ii),
}(\ref{Init11})\textit{ }and (\ref{Init12}). Therefore\textit{ }$\alpha$ is
divergent.\smallskip

\textit{Case 3.} $\left\vert \lambda_{1}\right\vert >\left\vert \lambda
_{2}\right\vert .$ Then the sequence $A_{np+p-1}/B_{np+p-1}$ converges to
$x_{1}$ by Lemma \ref{LemPower} \textit{(i) }and (\ref{Init12}). Moreover
for\textit{ }$0\leq q\leq p-2$ the sequence $A_{np+p-1}/B_{np+p-1}$ converges
to $x_{1}$ if $A_{q}-x_{2}B_{q}\neq0$ and to $x_{2}$ if $A_{q}-x_{2}B_{q}=0.$
As $x_{1}\neq x_{2}$ since $\lambda_{1}\neq\lambda_{2},$ $\alpha$ is
convergent if and only if $A_{q}-x_{2}B_{q}\neq0$ for all $0\leq q\leq p-2,$
and the proof of Theorem \ref{StolzTh} is complete.

\begin{remark}
In Case 3, the sequence $A_{n}/B_{n}$ has two different limit points
$x_{1}\neq x_{2}$ when $A_{q}-x_{2}B_{q}=0$ for some $q\in\left\{
0,\ldots,p-2\right\}  $ (\cite{Thiele},\cite[Satz 2.39]{Per2},\cite[Theorem
3.3 (B)]{JoTh}). This is known as Thiele's oscillations.
\end{remark}

\section{Proofs of Theorem \ref{ThGal} and Corollary \ref{CorGal}}

\label{SecGalois1}First we prove Theorem \ref{ThGal}. By Proposition
\ref{PropInv}, we have%
\[
A_{p}^{\prime}=A_{p},\quad B_{p}^{\prime}=A_{p-1},\quad A_{p-1}^{\prime}%
=B_{p},\quad B_{p-1}^{\prime}=B_{p-1}.
\]
Hence the matrix $M^{\prime}$ associated to $\alpha^{\prime}$ in Theorem
\ref{StolzTh} is%
\[
M^{\prime}:=\left(
\begin{array}
[c]{cc}%
A_{p-1}^{\prime} & A_{p}^{\prime}-b_{0}^{\prime}A_{p-1}^{\prime}\\
B_{p-1}^{\prime} & B_{p}^{\prime}-b_{0}^{\prime}B_{p-1}^{\prime}%
\end{array}
\right)  =\left(
\begin{array}
[c]{cc}%
B_{p} & A_{p}-b_{0}B_{p}\\
B_{p-1} & A_{p-1}-b_{0}B_{p-1}%
\end{array}
\right)  .
\]
Its trace is $\operatorname*{tr}M^{\prime}=B_{p}+A_{p-1}-b_{0}B_{p-1}%
=\operatorname*{tr}M$ and its determinant is%
\[
\det M^{\prime}=A_{p-1}B_{p}-A_{p}B_{p-1}=\det M.
\]
So the eigenvalues of $M^{\prime}$ are exactly $\lambda_{1}^{\prime}%
=\lambda_{1}$ and $\lambda_{2}^{\prime}=\lambda_{2}.$ Now let $x_{1}^{\prime}$
and $x_{2}^{\prime}$ be defined by%
\[
\lambda_{1}^{\prime}=\lambda_{1}=x_{1}^{\prime}B_{p-1}^{\prime}+a_{p}^{\prime
}B_{p-2}^{\prime},\quad\lambda_{2}^{\prime}=\lambda_{2}=x_{2}^{\prime}%
B_{p-1}^{\prime}+a_{p}^{\prime}B_{p-2}^{\prime}.
\]
Since $b_{p}^{\prime}=b_{p}=b_{0,}$ we see by (\ref{Init11}) and
(\ref{Init12}) that%
\begin{align*}
\left(  b_{0}-x_{1}\right)  B_{p-1}^{\prime}+a_{p}^{\prime}B_{p-2}^{\prime}
&  =\left(  b_{0}-x_{1}\right)  B_{p-1}^{\prime}+B_{p}^{\prime}-b_{0}%
B_{p-1}^{\prime}=\lambda_{2},\\
\left(  b_{0}-x_{2}\right)  B_{p-1}^{\prime}+a_{p}^{\prime}B_{p-2}^{\prime}
&  =\left(  b_{0}-x_{2}\right)  B_{p-1}^{\prime}+B_{p}^{\prime}-b_{0}%
B_{p-1}^{\prime}=\lambda_{1}.
\end{align*}
Therefore $x_{1}^{\prime}=b_{0}-x_{2}$ and $x_{2}^{\prime}=b_{0}-x_{1},$ so
that Theorem \ref{ThGal} follows immediately from Theorem \ref{StolzTh}%
.\medskip

Now we prove Corollary \ref{CorGal}. As $\alpha=x_{1}$ is irrational,
(\ref{Cond123}) cannot hold and the reverse continued fraction $\alpha
^{\prime}$ of $\alpha$ converges to $b_{0}-x_{2}$ by Theorem \ref{ThGal}.
Moreover $\lambda_{1}$ is irrational by (\ref{X1}). Hence $\lambda_{1}$ is
quadratic irrational and $\lambda_{2}$ is the conjugate of $\lambda_{1}$ since
the characteristic polynomial of $M$ has integer coefficients. Consequently
$x_{2}$ is the conjugate of $x_{1}=\alpha$ by (\ref{X2}), which yields
$\alpha^{\prime}=b_{0}-x_{2}=b_{0}-\alpha^{\ast}.$ Corollary \ref{CorGal} is proved.

\end{document}